\documentclass[11pt]{article}
\usepackage[final]{graphicx}
\usepackage[cmex10]{amsmath} 
\usepackage{amssymb}
\usepackage{amsfonts}
\usepackage{amsthm}
\usepackage{setspace}
\usepackage{epstopdf}
\usepackage{natbib}
\usepackage[affil-it]{authblk}
\usepackage[final,colorlinks]{hyperref}
\usepackage[margin=1in]{geometry}
\DeclareMathOperator{\Var}{Var}

\DeclareMathOperator{\num}{num}

\newtheoremstyle{nonum}{}{}{\itshape}{}{\bfseries}{.}{ }{\thmnote{#3}}
\theoremstyle{nonum}
\newtheorem{thm}{}
\newtheorem{lemma}{}
\newtheorem{proposition}{}

\begin{document}
\title{An elementary proof of convergence to the mean-field equations for an epidemic model}

\author{Benjamin Armbruster$^{\dag}$\footnote{Email addresses: armbrusterb@gmail.com (Benjamin Armbruster), ebeck@u.northwestern.edu (Ekkehard Beck)} $ $ 
and Ekkehard Beck$^{\dag}$}
\affil{$^{\dag}$Department of Industrial Engineering and Management Sciences, \\Northwestern University, Evanston, IL, 60208, USA \\}
\date{\today}

\maketitle

\begin{singlespacing}
\begin{abstract}
It is common to use a compartmental, fluid model described by a system of ordinary differential equations (ODEs) to model disease spread.  In addition to their simplicity, these models are also the mean-field approximations of more accurate stochastic models of disease spread on contact networks.  For the simplest case of a stochastic susceptible-infected-susceptible (SIS) process (infection with recovery) on a complete network, it has been shown that the fraction of infected nodes converges to the mean-field ODE as the number of nodes increases.  However the proofs are not simple, requiring sophisticated probability, partial differential equations (PDE), or infinite systems of ODEs.  We provide a short proof in this case for convergence in mean-square on finite time-intervals using a system of two ODEs and a moment inequality and also obtain the first lower bound on the expected fraction of infected nodes.
\end{abstract}
\end{singlespacing}

\section{Introduction} 
Interactions among individuals impact the transmission of infectious diseases, and the structure of the contact network imposes an important constraint on the transmission dynamics (\citealt{Keeling2005}). Models used in epidemiology to study disease transmission include mean-field models using ordinary differential equation (ODE) (\citealt{AndMay1991,Eames2002}), Markov or state-based models (\citealt{Bailey:1975}), and individual-based simulation models (\citealt{Goodreau2012,Keeling2005}).  These capture the contact network structure at different complexity levels and differ in analytical tractability, computational performance, and accuracy (\citealt{Keeling2000,Keeling2005,Bansal2007}). When choosing a particular type of model, whether mean-field models or individual-based simulations, the researcher should not only take into account the disease being transmitted, the size of the target population, the availability of data about the underlying network structure, the desired level of analytical tractability, and computational performance, but also the accuracy of the chosen model (\citealt{Bansal2007, Keeling2005b,Keeling2005}). 

The vast majority of studies, which assessed the accuracy of mean-field models, only relied on numerical experiments (\citealt{simon2010exact}).  However, the rigorous analysis of the accuracy of mean-field models requires deriving a mathematical link between mean-field approximations and the exact Markov models. \citet{Kurtz1970,Kurtz1971} are the first to show this (see also \citet{Ethier2005}).  They use operator semigroup and martingale techniques to prove that the exact models converge in probability over finite time-intervals to the solution of the mean-field ODE, using as a specific example the stochastic SIS-process on a complete network with increasing number of nodes.
\citet{simon2010exact} summarize the proof in \citet{Ethier2005} without going into the technical details.

\citet{simon2011exact} showed for the same setting as ours (i.e., a stochastic SIS-process on a complete network) based on the idea outlined in \citet{Diekmann2000}, that with increasing number of nodes the expected infected fraction converges to the solution of the mean-field ODE model using a partial differential equation (PDE) approach. This approach exploits the fact, that for a large number of nodes $n$, the discrete probability distribution $x_k(t)$ of the Markov model (i.e., $x_k(t)$ denotes the probability of having $k$ infected nodes) can be approximated using a continuous density function $\rho(t,k/n)$. The convergence is then proved by showing that the expected infected fraction of the resulting first-order PDE converges both to the solution of the mean-field model and to the expected fraction of the Markov model for large $n$.  A more detailed description of this proof is given in Chapter 2 in \citet{TaylorPhD2012}. 

Motivated by the first-order PDE approach and the stochastic approach, which both draw from multiple different mathematical areas, \citet{simon2010exact} introduced an ODE-based approach to show that the expected infected fraction of the Markov model converges uniformly on finite time-intervals to the solution of the mean-field ODE model.  Establishing an infinite homogeneous linear system of ODEs to approximate all moments of the probability distribution $x_k(t)$, \citet{simon2010exact} show using a perturbation theorem for the infinite system of ODEs that the mean-field solution is an upper bound of the expected infected fraction in the Markov model. Further, they provide a bound dependent on the network size $n$ for the absolute difference in the expected infected fraction between the mean-field model and the Markov model.  Note that neither the PDE nor the infinite ODE approach show that the distribution converges to the mean-field model.

In this paper, we provide a short proof showing that as the population size increases, the dynamics of the infected fraction of the stochastic SIS-process on a complete graph converges uniformly in mean-square on finite time-intervals to the mean-field ODE model.  We also provide the first lower bound for the expected fraction of infected nodes.  Note that convergence in mean-square implies convergence in probability and hence in distribution, thus providing a stronger result than previous proofs of \citet{Kurtz1970,Kurtz1971}, \citet{Ethier2005}, \citet{simon2011exact}, and  \citet{simon2010exact}.  More importantly, we only use a system of two ODEs, basic ODE techniques, and Jensen's inequality in our proof.  We hope our short proof using only elementary tools opens up the field to more applied researchers that might not be comfortable using martingale or PDE theory, provides the basis for further mean-field convergence results, and increases understanding about the performance of mean-field models.

\section{Mean-square convergence result} 
Consider a complete graph of $n$ nodes and let $X_n(t)$ be a Markov process describing which nodes in the network are infected at time $t$.  An infected node infects each neighbor at a rate $\tau/n$ and recovers at a rate $\gamma$.  We let $I_n(t):=\num_I (X_n(t))$ be the number of infected nodes and $i_n(t):=I_n(t)/n$ the infected fraction.  We may sometimes drop the dependence on $t$ and $n$. 

Our goal is to prove the following theorem.  It shows that the infected fraction converges in mean-square on finite time-intervals to the mean-field approximation as the population increases, and it gives upper and lower bounds on the expectation of the infected fraction.

\begin{thm}[Theorem 1]
If $i_n(0)=u$ for all $n$ where $u\in[0,1]$, then as $n\rightarrow\infty$,
$i_n(t)$ converges uniformly in mean-square, i.e., 
\[
E[|i_n(t)-y(t)|^2]\rightarrow 0
\]
on any time finite interval $[0,T]$ to the solution of the mean-field
equations: 
\begin{equation} \label{eq:1}
y'=\tau y(1-y)-\gamma y,\quad y(0)=u. 
\end{equation}
In addition we have the following bounds for $t\geq 0$,
\begin{subequations}\label{eq:bounds}
\begin{gather}
y(t)\geq E[i_n(t)]\geq z_{1,n}(t),\label{eq:boundsi}\\
z_{2,n}(t) \geq E[i_n(t)^2] \geq z_{1,n}(t)^2,
\end{gather}\end{subequations}
where $(z_{1,n}(t),z_{2,n}(t))$ solves the initial value problem:
\begin{subequations}\label{eq:zIVP}
\begin{align}
	z_{1,n}'&=\tau(z_{1,n}-z_{2,n})-\gamma z_{1,n},\\
	z_{2,n}'&=2\tau(z_{2,n}-z_{2,n}^{1.5})-2\gamma z_{2,n}+(1/n)(\tau+\gamma),\\
	&z_{1,n}(0)=u,\ z_{2,n}(0)=u^2.
\end{align}
\end{subequations}
\end{thm}

Of particular interest is the lower bound on $E[i_n]$ in \eqref{eq:boundsi}.  
All the bounds in \eqref{eq:bounds} are new, except for $y(t)\geq E[i_n(t)]$, which can also be found for example in the proof of Theorem 3.1 of \citet{Ganesh2005} or Proposition 5.3 of \citet{simon2010exact}, but even there, our approach is new.

\section{Proof of Theorem 1} 
Since the theorem holds trivially for $i_n(0)=u=0$, we now assume
$u>0$. To deal with another technicality, the initial value problems
\eqref{eq:1} and \eqref{eq:zIVP}
have unique solutions
because their right hand sides are smooth (i.e., continuously
differentiable). 
Since $y$ is deterministic, we will show mean-square
convergence by proving that $E[i_n(t)]\rightarrow y(t)$ and $E[i_n(t)^2]\rightarrow y(t)^2$
(or equivalently $\Var[i_n(t)]\rightarrow 0$), uniformly on $[0,T]$.
We show these limits by proving the bounds \eqref{eq:bounds} and that $z_{1,n}(t)\to y(t)$ and $z_{2,n}(t)\to y(t)^2$.

Let $SI(t):=\num_{SI}(X(t))$ be the number of edges in the network with
both a susceptible and infected endpoint. 
We start with the following standard proposition, which we prove in the appendix.

\begin{proposition}[Proposition 2]\label{prop2}
$E[I]'=(\tau/n)E[SI]-\gamma E[I]$ and 
\[
E[I^2]'=(\tau/n)E[SI((I+1)^2-I^2)]+\gamma E[I((I-1)^2-I^2)].
\]
\end{proposition}

In the case of a complete graph, $SI=I(n-I)$. Substituting into \nameref{prop2} and using $i=I/n$, we have 
\begin{subequations}
\begin{align}\label{eq:i}
E[i]'&=\tau(E[i]-E[i^2])-\gamma E[i],\\
\label{eq:i2}
E[i^2]'&=2\tau(E[i^2]-E[i^3])-2\gamma E[i^2]+(1/n)(\tau(E[i]-E[i^2])+\gamma E[i]).
\end{align}
\end{subequations}
 Applying Jensen's inequality, $E[i]^2\le E[i^2]$, to \eqref{eq:i},
\begin{equation}
E[i]'\le\tau E[i](1-E[i])-\gamma E[i].
\end{equation}

Then a standard ODE comparison theorem, \nameref{lem3}, justifies our intuition
that the solution to the mean-field equations \eqref{eq:1} is an upper
bound: $E[i(t)]\le y(t)$ for $t\ge0$.

\begin{lemma}[Lemma 3]\label{lem3}
Suppose that $a(t)$ and $b(t)$ are scalar; $f$ is smooth; $a(0)\le b(0)$;
and $a'(t)\le f(a(t))$ and $b'(t)=f(b(t))$ for $0\le t\le T$. Then
$a(t)\le b(t)$ for $0\le t\le T$.
\end{lemma}
\begin{proof}See any graduate ODE text such as Theorem 6.1 in \citet{hale2009ordinary}.\end{proof}

To obtain a lower bound we apply Jensen's inequality, $E[i^2]^{1.5}\le E[i^3]$
and $E[i]^2\le E[i^2]$, to \eqref{eq:i2},
\begin{equation}\label{eq:6}
E[i^2]'\le 2\tau(E[i^2]-E[i^2]^{1.5})-2\gamma E[i^2]+(1/n)(\tau E[i](1-E[i])+\gamma E[i]).
\end{equation}
Using the fact that $E[i]$ and $1-E[i]$ are in $[0,1]$, we obtain 
\begin{equation}\label{eq:10}
E[i^2]'\le 2\tau(E[i^2]-E[i^2]^{1.5})-2\gamma E[i^2]+(1/n)(\tau+\gamma).
\end{equation}
Applying \nameref{lem3}, we have $E[i_n(t)^2]\le z_{2,n}(t)$ for $t\geq 0$, where $z_{2,n}(t)$ solves the initial value problem:  
\begin{equation}\label{eq:11}
z_{2,n}'=2\tau(z_{2,n}-z_{2,n}^{1.5})-2\gamma z_{2,n}+(1/n)(\tau+\gamma),\quad z_{2,n}(0)=u^2.
\end{equation}
Now going back to \eqref{eq:i}, we have a lower bound for $E[i]'$,
\begin{equation}\label{eq:12}
E[i]'\ge\tau(E[i]-z_{2,n})-\gamma E[i]
\end{equation}
Applying \nameref{lem3} here by treating $z_{2,n}(t)$ as an exogenous function, we have
$E[i_n(t)]\ge z_{1,n}(t)$ for $t\geq 0$, where $z_{1,n}(t)$ solves the initial
value problem: 
\begin{equation}\label{eq:13}
z_{1,n}'=\tau(z_{1,n}-z_{2,n})-\gamma z_{1,n},\quad z_{1,n}(0)=u.
\end{equation} 
Our last lower bound $E[i_n^2]\geq E[i_n]^2\geq z_{1,n}^2$ then follows from Jensen's inequality.

Having proven the bounds,
the next step is to show that $z_{1,n}(t)\rightarrow y(t)$ and $z_{2,n}(t)\rightarrow y(t)^2$
uniformly on $[0,T]$. 
Note that $\bar{z}=(y,y^2)$ is a solution to the system,
\begin{subequations}
\label{eq:8}
\begin{align}
& \bar{z}'_1=\tau(\bar{z}_1-\bar{z}_2)-\gamma \bar{z}_1, \label{eq:81} \\
& \bar{z}'_2=2\tau(\bar{z}_2-\bar{z}_2^{1.5})-2\gamma \bar{z}_2, \label{eq:82}\\
& \bar{z}'_1=u,\quad  \bar{z}_2=u^2, \label{eq:83}
\end{align}
\end{subequations}
because \eqref{eq:81} and \eqref{eq:83} become \eqref{eq:1}
and the right hand side of \eqref{eq:82} ends up being $2yy'$, which equals $(y^2)'$ as we require.
Uniqueness holds because the right hand side is smooth.
Substituting $\epsilon=1/n$, the right hand side of \eqref{eq:zIVP} is also smooth in $\epsilon$ and it converges uniformly to the right hand side of \eqref{eq:8} 
as $\epsilon\to 0$ (i.e., $n\to\infty$).
Then by the
results about the continuous dependence on parameters, the ODE solutions converge, $(z_{1,n},z_{2,n})\to (y,y^2)$ uniformly on $[0,T]$, proving the claim.

\section{Conclusion} 

We gave a short proof showing that as the population size increases, the dynamics of the infected fraction of the stochastic SIS-process on a complete graph converges uniformly in mean-square on finite time-intervals to the mean-field ODE model. Using the ODE from \nameref{prop2} to describe the dynamics of the first moment (the expected number of nodes infected) and Jensen's inequality to approximate the second moment in this equation, we showed that the solution of the mean-field ODE model provides an upper bound to the expected infected fraction. Further, we established the first lower bound on the expected fraction infected.  The lower bound is from a system of two ODEs: the ODE describing the dynamics of the first moment combined with an ODE describing the dynamics of the second moment, which we bounded using Jensen's inequality. Finally, uniform convergence in mean-square was shown by proving that the lower bound (i.e., the solution of the system of two ODEs approximating the dynamics of the first two moments) converges to the infected fraction and its square in the mean-field ODE model. Thus, the approach presented provides both a stronger convergence result compared to previous approaches showing convergence in probability or convergence of the expected value as well as a shorter, more elementary proof.

Mean-field models appear in many fields, not just epidemic modeling, and if such an elementary approach to proving convergence is generalizable, it would have many users.
However, one limitation of our approach is that \nameref{lem3} about ODE inequalities only holds for scalar ODEs and the argument we used relied on the second equation not depending on the first.  In cases where a mean-field model leads to a bound involving a system of more and less simple ODEs, some other argument may be needed.  Developing such arguments to generalize this elementary approach beyond the SIS-model to more complicated epidemic models is an area for future work.

\paragraph{Acknowledgements.} We thank Peter L. Simon, Istvan Z. Kiss, and two anonymous reviewers for helpful comments. 

\pagebreak
\bibliography{references}

\begin{thebibliography}{}

\bibitem[\protect\astroncite{Anderson and May}{1991}]{AndMay1991}
Anderson, R.~M. and May, R.~M. (1991).
\newblock {\em Infectious {D}iseases of {H}umans {D}ynamics and {C}ontrol}.
\newblock Oxford University Press.

\bibitem[\protect\astroncite{Bailey}{1975}]{Bailey:1975}
Bailey, N. (1975).
\newblock {\em The Mathematical Theory of Infectious Diseases and its
  Applications}.
\newblock Griffin, London.

\bibitem[\protect\astroncite{Bansal et~al.}{2007}]{Bansal2007}
Bansal, S., Grenfell, B.~T., and Meyers, L.~A. (2007).
\newblock When individual behaviour matters: homogeneous and network models in
  epidemiology.
\newblock {\em Journal of The Royal Society Interface}, 4(16):879--891.

\bibitem[\protect\astroncite{Diekmann and Heesterbeek}{2000}]{Diekmann2000}
Diekmann, O. and Heesterbeek, J. A.~P. (2000).
\newblock {\em Mathematical Epidemiology of Infectious Diseases: Model
  Building, Analysis and Interpretation}.
\newblock Wiley, Chichester, {UK}.

\bibitem[\protect\astroncite{Eames and Keeling}{2002}]{Eames2002}
Eames, K. T.~D. and Keeling, M.~J. (2002).
\newblock Modeling dynamic and network heterogeneities in the spread of
  sexually transmitted diseases.
\newblock {\em Proceedings of the National Academy of Sciences},
  99(20):13330--13335.

\bibitem[\protect\astroncite{Ethier and Kurtz}{2005}]{Ethier2005}
Ethier, S.~N. and Kurtz, T.~G. (2005).
\newblock {\em Markov {P}rocesses: {C}haracterization and {C}onvergence}.
\newblock Wiley series in probability and statistics. Wiley.

\bibitem[\protect\astroncite{Ganesh et~al.}{2005}]{Ganesh2005}
Ganesh, A., Massoulie, L., and Towsley, D. (2005).
\newblock The effect of network topology on the spread of epidemics.
\newblock In {\em INFOCOM 2005. 24th Annual Joint Conference of the IEEE
  Computer and Communications Societies. Proceedings IEEE}, volume~2, pages
  1455--1466.

\bibitem[\protect\astroncite{Goodreau et~al.}{2012}]{Goodreau2012}
Goodreau, S., Carnegie, N., Vittinghoff, E., Lama, J., Sanchez, J., et~al.
  (2012).
\newblock {W}hat drives the {US} and {P}eruvian {HIV} epidemics in men who have
  sex with men {(MSM)}?
\newblock {\em PLoS ONE}, 7(11):e50522.

\bibitem[\protect\astroncite{Hale}{2009}]{hale2009ordinary}
Hale, J. (2009).
\newblock {\em Ordinary Differential Equations}.
\newblock Dover Books on Mathematics Series. Dover Publications.

\bibitem[\protect\astroncite{Keeling}{2005}]{Keeling2005b}
Keeling, M. (2005).
\newblock The implications of network structure for epidemic dynamics.
\newblock {\em Theoretical Population Biology}, 67(1):1--8.

\bibitem[\protect\astroncite{Keeling and Eames}{2005}]{Keeling2005}
Keeling, M.~J. and Eames, K. T.~D. (2005).
\newblock Networks and epidemic models.
\newblock {\em J. R. Soc. Interface}, 2:295--307.

\bibitem[\protect\astroncite{Keeling and Grenfell}{2000}]{Keeling2000}
Keeling, M.~J. and Grenfell, B.~T. (2000).
\newblock Individual-based perspectives on {$R_0$}.
\newblock {\em Journal of Theoretical Biology}, 203(1):51--61.

\bibitem[\protect\astroncite{Kurtz}{1970}]{Kurtz1970}
Kurtz, T.~G. (1970).
\newblock Solutions of ordinary differential equations as limits of pure jump
  {M}arkov processes.
\newblock {\em Journal of Applied Probability}, 7:49--58.

\bibitem[\protect\astroncite{Kurtz}{1971}]{Kurtz1971}
Kurtz, T.~G. (1971).
\newblock Limit theorems for sequences of jump {M}arkov processes approximating
  ordinary differential processes.
\newblock {\em Journal of Applied Probability}, 8(2):344--356.

\bibitem[\protect\astroncite{Rand}{1999}]{Rand1999}
Rand, D.~A. (1999).
\newblock {C}orrelation equations and pair approximations for spatial
  ecologies.
\newblock {\em CWI Quarterly}, 12(3\&4):329--368.

\bibitem[\protect\astroncite{Simon and Kiss}{2012}]{simon2010exact}
Simon, P.~L. and Kiss, I.~Z. (2012).
\newblock From exact stochastic to mean-field {ODE} models: a case study of
  three different approaches to prove convergence results.
\newblock {\em IMA Journal of Applied Mathematics}, pages 1--20.

\bibitem[\protect\astroncite{Simon et~al.}{2011}]{simon2011exact}
Simon, P.~L., Taylor, M., and Kiss, I.~Z. (2011).
\newblock Exact epidemic models on graphs using graph-automorphism driven
  lumping.
\newblock {\em Journal of {M}athematical {B}iology}, 62(4):479--508.

\bibitem[\protect\astroncite{Taylor}{2012}]{TaylorPhD2012}
Taylor, M. (2012).
\newblock {\em {E}xact and {A}pproximate {E}pidemic {M}odels on {N}etworks:
  {T}heory and {A}pplications.}
\newblock PhD thesis, University of Sussex.

\end{thebibliography}
\pagebreak
\appendix
\section*{Appendix}

Our proof of Proposition 2 is shorter and more elementary than the original proof by \citet{Rand1999}, but see also \citet{simon2011exact} for a more formulaic proof.

\begin{proposition}[Proposition 2]
$E[I]'=(\tau/n)E[SI]-\gamma E[I]$ and 
\[
E[I^2]'=(\tau/n)E[SI((I+1)^2-I^2)]+\gamma E[I((I-1)^2-I^2)].
\]
\end{proposition}
\begin{proof} In network configuration $x$, the system can either transition
to states $x^+$ with $\num_I(x^+)=\num_I(x)+1$ or to states $x^-$ with
$\num_I(x^-)=\num_I(x)-1$ (i.e., only to states with $\pm 1$ infected
nodes). The aggregate rate for the first is $\gamma \num_I(x)$ and
$(\tau/n)\num_{SI}(x)$ for the second, proving that
\begin{equation}\label{eq:prop2a} \lim_{h\rightarrow 0} (E[I(h)|I(0)=x]-\num_I(x))/h=(\tau/n)\num_{SI}(x)-\gamma \num_I(x).
\end{equation}
The transition
rates and times of $I(t)^2$ are the same as those of $I(t)$, just now the jump sizes are $(I\pm1)^2-I^2$
instead of $\pm1$:
\begin{multline}\label{eq:prop2b}
\lim_{h\rightarrow 0} (E[I(h)^2|I(0)=x]-\num_I(x)^2)/h\\
=(\tau/n)\num_{SI}(x)((x+1)^2-x^2)+\gamma \num_I(x)((x-1)^2-x^2).
\end{multline}
Multiplying both sides of \eqref{eq:prop2a} and \eqref{eq:prop2b}
by $P[X(0)=x]$ and summing over $x$ (unproblematic since it is a finite sum) proves the claims.
\end{proof}
\end{document}